\documentclass[10pt]{amsart}
\usepackage[english]{babel}

\usepackage{latexsym}
\usepackage{amsmath}
\usepackage{amsfonts}
\usepackage{euscript}
\usepackage[all]{xy}
\usepackage{latexsym,tabularx}
\usepackage{amscd}
\usepackage{amssymb}
\usepackage{epsfig}
\newtheorem{thm}{Theorem}[section]
\newtheorem{lem}{Lemma}[section]

\newtheorem{prop}{Proposition}[section]

\begin{document}
\title{An stratification of $B^4(2,K_C)$ over a general curve.}
\author{Abel Castorena and Graciela Reyes-Ahumada}
\address{Centro de Ciencias Matem\'aticas (Universidad Nacional Aut\'onoma de M\'exico), Campus Morelia,
Apdo. Postal 61-3(Xangari). C.P. 58089, 
Morelia, Michoac\'an. M\'exico. 
Instituto de Matem\'aticas (Universidad Nacional Aut\'onoma de M\'exico), Sede Oaxaca,
Leon 2 Centro,  C.P. 68000,
Oaxaca, Oaxaca. M\'exico.}
\email{abel@matmor.unam.mx, grace@matmor.unam.mx}

\begin{abstract} For a general curve $C$ of genus $g\geq 10$, we show that the Brill-Noether locus $B^4(2,K_C)$ contains irreducible sub-varieties $\mathcal B_3\supset \mathcal B_4\supset \cdots \supset \mathcal B_n$, where $\mathcal B_n$ is of dimension $3g-10-n$ and $\mathcal B_3$ is an irreducible component of the expected dimension $3g-13$.\end{abstract}
\thanks{The first author was supported by PAPIIT IN100716, Universidad Nacional Aut\'onoma de M\'exico).\\
Second author was supported by a FORDECYT(CONACyT, M\'exico) fellowship.}

\maketitle
\section{Introduction}

\noindent We study the Brill-Noether locus $B^4(2,K_C)$ that parametrizes classes of stable rank two vector bundles with canonical determinant and four sections over a general irreducible complex projective curve $C$.

 \noindent In \cite{CR}, the authors construct an irreducible component of $B^4(2,K_C)$ of the expected dimension by studying the space of extensions of line bundles following the spirit in \cite{CF1}. In this work we generalize such construction over the general curve for the case of rank two bundles with four sections in the following sense: we exhibit a chain of irreducible closed subloci $\mathcal B_3\supset\cdots\supset \mathcal B_n$ inside $B^4(2,K_C)$ where $\mathcal B_3$ is an irreducible component of expected dimension $3g-13$ ($\mathcal B_3$ is the component constructed in \cite{CR}) and the dimension of any stratrum $\mathcal B_n$ is $3g-10-n$. In order to give such stratification we analize the space $Ext^1:=Ext^1(K_C(-D_n),\mathcal{O}_C(D_n))$ which parametrizes extensions of the form
$$u:0\rightarrow \mathcal{O}_C(D_n)\rightarrow E_u \rightarrow K_C(-D_n)\rightarrow 0;$$
where $D_n$ is a general effective divisor of degree $n$ over $C$. We construct the locus $\mathcal B_n$ as the image of an irreducible determinantal locus $\Delta_n \subset \mathbb{P}Ext^1$ through the rational map $\pi:\Delta_n \rightarrow U(2,K_C), [u]\to [E_u]$, where $U(2,K_C)$ denote the moduli space of rank two stable bundles with canonical determinant. By looking at the deformations inside the global space of extensions constructed in \cite{CF1} and let the divisor $D_n$ to vary, we conclude that such $\mathcal B_n$ stratify the component $\mathcal B_3$ of $B^4(2,K_C)$. Our main result is Theorem \ref{mainthm}, and we prove it by applying from Lemma 3.2 to Lemma 3.5. With these Lemmas we construct an stratification of $B^4(2,K_C)$ by the loci $\mathcal B_n$.  

\noindent We would like to emphasize that in this work we do not find a new component of the Brill-Noether locus $B^4(2,K_C)$, but we do give a partial answer to the problem to describe deformations of a regular component by studying deformations of extensions of line bundles. An interesting question (which remains open) coming  from this stratification is to describe for every $n\geq 3$, which kind of divisorial component defines $\mathcal B_{n+1}$ inside $\mathcal B_n$ at least for curves of small genus .

\section{Extensions of Line Bundles and Geometrically Ruled Surfaces.}
\subsection{Global space of extensions}\label{extensions} We restrict our attention to special line bundles $L, N$ over a smooth curve $C$, and we assume that for any integer $t\geq 1$, $h^1(N)\geq t$ and $h^0(L)\geq\text{max}\{1,h^1(N)-t\}$. For a more general context, see \cite{CF1}. 

\noindent Given a line bundle $M$ over a smooth curve, we denote by $\rho_M$ (or simply $\rho$, when $M$ is understood) the Brill-Noether number of $M$. 

\noindent For any $u\in H^1(N\otimes L^{\vee})\simeq Ext^1(L,N)$, consider the coboundary map $\partial_u:H^0(L)\rightarrow H^1(N)$ and denote by $\text{coker}(\partial_u):=\text{dim(Cokernel}(\partial_u))$. Consider the cup product map $\cup:H^0(L)\otimes H^1(N\otimes L^{\vee})\to H^1(N)$. We have that $\partial_u(s)=s\cup u$. By Serre duality, $\cup$ is equivalent to consider
for any subspace $W\subseteq H^0(K_C\otimes N^{\vee})$, the multiplication map $\mu_W:H^0(L)\otimes W\to H^0(K_C\otimes L\otimes N^{\vee})$. 

\noindent For $u\in H^1(N\otimes L^{\vee})\simeq Ext^1(L,N)\simeq H^0(C,K_C\otimes L\otimes N^{\vee})^{\vee}$, we denote by $H_u:=\{u=0\}\subset H^0(C, K_C\otimes L\otimes N^{\vee})$ the hyperplane defined by $u$, so if $u\in Ext^1(L,N)$  is such that $coker(\partial_u)\geq t$, then $W:=\text{Im}(\partial_u)^{\perp}\subset H^0(K_C-N)$ satisfies that $Im(\mu_{W})\subset H_u$, thus we define:

$$\mathcal W_t=\{u\in Ext^1(L,N)| \exists W\subset H^0(K_C-N)\text{ with }\text{dim}(W)\geq t\text{ and } Im(\mu_W)\subset H_u\}.$$

\noindent Note that $\mathcal W_t$ has a natural structure of determinantal variety. 

\noindent Let $C$ be a general curve of genus $g\geq 3$. Let $0\to N\to E\to L\to 0$ be an exact sequence such that
$\text{deg}(N)=d-\delta>0, \text{deg}(L)=\delta>0$ and $h^0(L)\cdot h^1(L)>0,h^0(N)\cdot h^1(N)>0$. Let 
$\ell:=h^0(L), j:= h^1(L), n:=h^0(N), r:=h^1(N)$. Set 
$\mathcal Y:=\text{Pic}^{d-\delta}( C)\times W^{\ell-1}_{\delta}( C),
\hskip2mm Z:=W^{n-1}_{d-\delta}( C)\times W^{\ell-1}_{\delta} (C)\subset\mathcal Y$. One has 
$\text{dim}(\mathcal Y)=g+\rho_L$ and 
$\text{dim}(Z)=\rho_L+\rho_N$. Since $C$ is general, we have that $\mathcal Y$ and $Z$ are irreducible when $\rho>0$; otherwise
one replace the (reducible) zero-dimensional Brill-Noether locus with one of its irreducible components to construct 
$\mathcal Y$ and $Z$ as above.

\noindent When $2\delta-d\geq 1$ one can construct a vector bundle $\mathcal E\to\mathcal  Y$ of rank
$m=2\delta-d+g-1$ (see \cite{ACGH}, p. 176-180), together with a projective bundle morphism 
$\gamma: \mathbb P(\mathcal E)\to\mathcal Y$ where, for $y=(N,L)\in\mathcal Y$, the fiber 
$\gamma^{-1}(y)=\mathbb P(\text{Ext}^1(N,L))=\mathbb P$. We have that 
$\text{dim}(\mathbb P(\mathcal E))=\text{dim}(\mathcal Y)+m-1\text{ and }\text{dim}(\mathbb P(\mathcal E)|_Z)=
\text{dim}(Z)+m-1$. Since (semi)stability is an open condition, for $2\delta-d\geq 2$ there is an open, dense subset
$\mathbb P(\mathcal E)^0\subseteq\mathbb P(\mathcal E)$ and a morphism $\pi_{d,\delta}:\mathbb P(\mathcal E)^0\to U_C(d)$.  
In \cite{CR} the authors showed that on a general curve $C$ of genus $g\geq 8$, there exists an irreducible 
component of $\mathcal W_3$, $\Delta_3\subset\mathbb P(\mathcal E)$ of dimension $3g-13$. Moreover, this component fills up an irreducible component 
$\mathcal B_3\subseteq B^4(2,K_C)$ of dimension $3g-13$. 

\noindent Following the above construction, We will show in the next section that for $n\geq 4$, there exist closed and irreducible subschemes $\mathcal B_n$ all of them contained in $\mathcal B_3$ such that $\mathcal B_{n+1}\subset\mathcal B_n$ and $\text{dim}(\mathcal B_n)=3g-10-n$.

\subsection{Unisecant curves in geometrically ruled surfaces} Let $E$ be a rank-two vector bundle over a smooth irreducible complex projective 
curve $C$ of genus $g$. Let $S:=\mathbb P(E)$ be the (geometrically) ruled surface associated to it, with structure map $p:S\to C$.  For any $x\in C$ we denote by $f_x=p^{-1}(x)\simeq\mathbb P^1$ and we write $f_D=p^*(D)$ for $D\in\text{Div}(C) $. We denote by $f$ a general
fiber of $p$ and by $\mathcal O_S(1)$ the tautological line bundle on $S$. We recall that there is a one-to-one correspondence between sections $\Gamma$ of $S$ and surjective maps $E\to\!\!\to L$ with $L$ a line bundle on $C$ (cf. \cite{H}, chapter V). One has an exact sequence
$0\to N\to E\to L\to 0$, where $N\in\text{Pic}( C)$. The surjection $E\to\!\!\to L$ induces an inclusion 
$\Gamma=\mathbb P(L)\subset S=\mathbb P(E)$. If $L=\mathcal O_C(B), B\in\text{Div}( C)$ with $b=\text{deg}(B)$, then 
$b=H\cdot\Gamma$ and $\Gamma\sim H+f_N$ where $N=L\otimes\text{det}(E)^{\vee}\in\text{Pic} (C)$. For a section 
$\Gamma$ that corresponds to the exact sequence $0\to N\to E\to L\to 0$, we have that $\mathcal N_{\Gamma/S}$, the 
normal sheaf of $\Gamma$ in $S$, is such that 
$\mathcal N_{\Gamma/S}=\mathcal O_{\Gamma}(\Gamma)\simeq N^{\vee}\otimes L$. In particular 
$\text{deg}(\Gamma)=\Gamma^2=\text{deg}(L)-\text{deg}(N)$. If $\Gamma\sim H-f_D$, i.e. when $N=\mathcal O_C(D)$, 
then $|\mathcal O_S(\Gamma)|\simeq \mathbb P(H^0(E(-D)))$.

\noindent An element $\Gamma\in |H+f_D|$ is called {\it{unisecant curve}}, and irreducible unisecants are called {\it{sections}} of $S$. For any $m\in\mathbb N$, denote by $\text{Div}^{1,m}(S)$ 
the Hilbert scheme of unisecant curves of $S$ of degree $m$ with respect to $\mathcal O_S(1)$. Since 
elements of $\text{Div}^{1,m}(S)$ correspond to quotients of $E$,  therefore $\text{Div}^{1,m}(S)$ can be endowed with
a natural structure of Quot-Scheme (cf. \cite{S}, section 4.4), and one has an isomorphism (see e.g. \cite{CF1}, Section 2.4.) $$\Phi_{1,m}:\text{Div}^{1,m}(S)\xrightarrow{\simeq}\text{Quot}^C_{E,m-g+1},\hskip3mm\Gamma\to (E\to\!\!\to L)$$

\noindent Let $\Gamma\in\text{Div}^{1,m}(S)$ be. We say that:

\noindent (i). $\Gamma$ is {\it{linearly isolated}} if $\text{dim}(|\mathcal O_S(\Gamma)|)=0$;

\noindent (ii) $\Gamma$ is called {\it{special unisecant}} if $h^1(\Gamma,\mathcal O_S(1)\otimes\mathcal O_{\Gamma})>0$. 

\noindent Inside $\text{Div}^{1,m}(S)$, we consider the scheme $\mathcal S^{1,n}$ that parametrizes degree $n$, \textit{special unisecant curves of} $S$. 

\noindent Let $\mathcal F\subset\text{Div}^{1,n}(S)$ be a subscheme, and let $\Gamma$ be a special unisecant of $S$. Assume that $\Gamma\in\mathcal F$, where 
$\mathcal F\subset\text{Div}^{1,n}(S)$, $\Gamma$ is called {\it{specially isolated in $\mathcal F$}} if $\text{dim}_{\Gamma}(\mathcal F\cap\mathcal S^{1,n})=0$.

\section{Construction of components in the space of extensions.}
The following construction is a generalization of the one given in \cite{CR}, and this is an adapted argument due to Robert Lazarsfeld. 

\begin{lem} Fix an integer $n\geq 4$. Let $C$ be a non-hyperelliptic curve of genus $g\geq n+5$. Let $D_n$ be a general effective divisor of degree $n$ over $C$. There exists a vector bundle $E\rightarrow C$ satisfying the following properties:
\begin{description}
 \item[i] $rank(E)=2$;
 \item[ii] $det(E)=K_C(-D_n)$;
 \item[iii] $h^0(E)=3$;
 \item[iv] $E$ is globally generated;
 \item[v] $h^0(E^{\vee})=0$.
\end{description}
\end{lem}

\begin{proof}
Set $k=g-2-n$ and let $D_k$ be a general effective divisor of degree $k$. Let $Ext^1(K_C(-D_n-D_k),\mathcal{O}_C(D_k))$ be the space parametrizing extensions of the form
$$u:0\rightarrow \mathcal{O}_C(D_k)\rightarrow E\rightarrow K_C(-D_n-D_k) \rightarrow 0.$$

\noindent Since $Ext^1(K_C(-D_n-D_k),\mathcal{O}_C(D_k))=H^1(D_k-K_C+D_n+D_k)=H^0(2K_C-D_n-2D_k)^{\vee}$, then $dim(Ext^1(K_C(-D_n-D_k),\mathcal{O}_C(D_k)))=3g-3-n-2k=3g-3-n-2g+4+2n=g+n+1$. Note that for such extensions $u$, the corresponding bundles $E$ satisfy $rank(E)=2$ and $det(E)=K_C(-D_n)$. Consider the coboundary map  
$$\partial: Ext^1(K_C(-D_n-D_k),\mathcal{O}_C(D_k)) \rightarrow H^0(K_C(-D_n-D_k))^{\vee}\otimes H^1(\mathcal{O}_C (D_k))$$ which associates to an extension $u$ its coboundary map in cohomology $$\partial_u: H^0(K_C(-D_n-D_k))\rightarrow H^1(\mathcal{O}_C(D_k)).$$

\noindent For these extensions $u\in Ker(\partial)$ we have that $h^0(E)=3$. Since the map $\partial$ is dual to the map
$$\mu:H^0(K_C(-D_n-D_k))\otimes H^0(K_C(-D_k))\rightarrow Ext^1(K_C(-D_n-D_k),\mathcal{O}_C(D_k))^{\vee},$$

\noindent then $dim(Ker(\partial))=coker(\mu)$. By the base-point-free pencil trick to the pencil defined by $K_C(-D_n-D_k)$, we have that $Ker(\mu)\cong H^0(K_C(-D_k)-(K_C(-D_n-D_k)))=H^0(\mathcal{O}_C(D_n))$, since $h^0(\mathcal{O}_C(D_n))=1$, then $dim(Ker(\mu))=1$ and $$dim(Ker(\partial))=coker(\mu)=(g+n+1)-(2(g-k)-1)=g-n-2=k.$$ 

\noindent Consider now $u\in Ker(\partial)$. For every $x\in C$, $h^0(E(-x))\geq 1$. If $x\notin Supp(D_k)$, then tensoring $u$ by $\mathcal{O}(-x)$ we get
$$0\rightarrow \mathcal{O}_C(D_k-x)\rightarrow E(-x)\rightarrow K_C(-D_n-D_k-x) \rightarrow 0.$$
Then $h^0(E(-x))\leq h^0(\mathcal{O}_C(D_k-x))+h^0(K_C(-D_n-D_k-x))=1$ and the equality holds. Then $E$ could fail to be generated at most at the points in $Supp(D_k)$.

\noindent Suppose $E$ is not generated at $x\in Supp(D_k)$, then the extension $u$ induces an extension
$$u':0\rightarrow \mathcal{O}_C(D_k-x)\rightarrow E'\rightarrow K_C(-D_n-D_k) \rightarrow 0,$$

\noindent such that $$u'\in Ker(\partial':Ext^1(K_C(-D_n-D_k),\mathcal{O}_C(D_k-x)) \rightarrow H^0(K_C(-D_n-D_k))^{\vee}\otimes H^1(\mathcal{O}_C(D_k-x)).$$ 

\noindent Note that $Ext^1(K_C(-D_n-D_k),\mathcal{O}_C(D_k-x))=H^1(D_k-x-K_C+D_n+D_k)=H^0(2K_C(-D_n-2D_k+x))^{\vee}$, then $dim(Ext^1(K_C(-D_n-D_k),\mathcal{O}_C(D_k-x)))=h^0(2K_C-D_n-2D_k+x)=3g-3-n-2k+1=3g-3-n-2g+4+2n+1=g+n+2$.

\noindent The coboundary map $\partial'$ is dual to
$$\mu':H^0(K_C(-D_n-D_k))\otimes H^0(K_C(-D_k+x))\rightarrow Ext^1(K_C(-D_n-D_k),\mathcal{O}_C(D_k-x))^{\vee},$$

\noindent then $dim(Ker(\partial'))=coker(\mu)$ and by the base-point-free pencil trick, $Ker(\mu')\cong H^0(K_C(-D_k+x)-K_C(-D_n-D_k))=H^0(\mathcal{O}_C(D_n+x))$, then $dim(Ker(\mu'))=1$ and $$dim(Ker(\partial'))=coker(\mu')=(g+n+2)-(2(g-k+1)-1)=g-n-3,$$
\noindent thus, the locus of bundles which are not generated has codimension at least one inside $Ker(\partial)$, so there exists a bundle $E$ which is globally generated.

\noindent Suppose there is a non-zero section $s\in H^0(E^{\vee})$. Then $s$ induces an injection $\mathcal{O}_C\hookrightarrow E^{\vee}$ by tensoring by $K_C(-D_n)$ and using the isomorphism $E\cong E^{\vee}\otimes K_C(-D_n)$ we get
$K_C(-D_n)\hookrightarrow E$ and $h^0(K_C(-D_n))\leq h^0(E)$ giving a contradiction with our hypothesis on the genus.
\end{proof}
     
\begin{prop}
Fix an integer $n\geq 3$. Let $C$ be a non-hyperelliptic curve of genus $g\geq n+5$, and let $D_n$ be a general effective divisor of degree $n$ over $C$. For the general $V\in G(3,H^0(K_C-D_n))$, the multiplication map
$$\mu_V:V\otimes H^0(K_C-D_n)\rightarrow H^0(2K_C-2D_n)$$
has kernel of dimension three.
\end{prop}
\begin{proof}
By uppersemicontinuity it is enough to show there exists an element $V\in G(3,H^0(K_C-D_n))$ such that $dim(Ker(\mu_V))=3$. Consider the bundle $E$ from the previous Lemma. Since $E$ is globally generated we have the exact sequence
$$0\rightarrow (K_C(-D_n))^{\vee}\rightarrow H^0(E)\otimes \mathcal{O}_C\rightarrow E\rightarrow 0,$$
by dualizing this sequence and computing the cohomology, we get an injective map $H^0(E)^{\vee}\hookrightarrow H^0(K_C-D_n)$. Denote by $V$ the image of this injective map. By twisting the dual of the above sequence by $K_C-D_n$ and identifying $E^{\vee}\otimes K_C(-D_n)\cong E$, we have a cohomology sequence
$$0\rightarrow H^0(E)\rightarrow V\otimes H^0(K_C-D_n)\rightarrow H^0(2K_C-2D_n)$$
and the kernel of $\mu_V$ is identified with $H^0(E)$, which is of dimension three. This completes the proof.
\end{proof}

Let $X$ be a smooth curve contained in a projective space $\mathbb P$. For any positive integer $h$, denote by $\text{Sec}_h(X)$, the $h^{th}$-secant variety of $X$, defined as the closure of the union of all linear subspaces $<\phi(D)>\subset\mathbb P$, for all effective general divisors of degree $h$. One has that $\text{dim (Sec}_h(X))=\text{min}\{\text{dim}(\mathbb P),2h-1\}$ (see \cite{L}). 

\begin{thm}\label{mainthm}
Let $C$ be a non-hyperelliptic curve of genus $g\geq 10$ with general moduli. Let $3 \leq n\leq g-4$ be an integer. There exist closed irreducible loci $\mathcal B_n \subset B^4(2,K_C)$ of dimension $3g-10-n$ whose general point corresponds to a stable bundle $E$ fitting in an exact sequence of the form
$$0\rightarrow \mathcal{O}_C(D_n)\rightarrow E \rightarrow K_C(-D_n)\rightarrow 0,$$
where $D_n$ is a general effective divisor of degree $n$. In particular for $n=3$, $\mathcal B_3$ is a regular component of $B^4(2,K_C)$ and the loci $B_3\supset B_4\supset\cdots \supset B_n$ stratify $B_3$.\end{thm}

We recall that the component $\mathcal B_3$ was constructed in \cite{CR}, thus we need only to construct the components $\mathcal B_n$ for $n\geq 4$. To do this we start by showing the following lemma

\begin{lem} For every integer $n\geq 3$, inside the locus $\mathcal W_3$ there exist loci 
$\Delta_n\subset\mathcal W_3$, such that $\Delta_n$ is an irreducible component of dimension $3g-10-2n$.
\end{lem}

\begin{proof} Consider the space of extensions $Ext^1(K_C(-D_n),\mathcal{O}_C(D_n))$ which parametrizes exact sequences of the form
$$u:0\rightarrow \mathcal{O}_C(D_n)\rightarrow E \rightarrow K_C(-D_n)\rightarrow 0.$$
Set $\mathbb{P}:=\mathbb{P}Ext^1(K_C(-D_n),\mathcal{O}_C(D_n))$ and
consider the incidence variety 
$$\mathcal{J}=\{V\times u| Im(\mu_V)\subset H_u\}\subset G(3,H^0(K_C-D_n))\times \mathbb{P}.$$ 

Let $\pi_1:\mathcal{J}\rightarrow G(3,H^0(K_C-D_n))$ and $\pi_2:\mathcal{J}\rightarrow \mathbb{P}$ be the projections. By Proposition 3.1, for the general element $V\in G(3,H^0(K_C-D_n))$, 
the codimension of $Im(\mu_V)$ inside $H^0(2K_C-2D_n)$ is equal to $n$, then $Im(\mu_V)$ is contained in a hyperplane, and $\mathcal{J}$ is not empty. Then there exists an irreducible 
component $J_n\subset \mathcal{J}$ dominating the Grassmannian through $\pi_1$. The general fiber of $\pi_1|_{J_n}$ is 
isomorphic to the linear system of hyperplanes containing $Im(\mu_V)$, so the general fiber of $\pi_1|_{J_n}$ is of dimension $n-1$ and $dim(J_n)=3g-2n-10$. Set $\Lambda_n:=\pi_2(J_n)\subset \mathbb{P}$, and consider the image $X$ of the map $C\hookrightarrow \mathbb{P}$ induced by the linear system $|(K_C(-D_n))^{\otimes 2}|$. Note that for $\epsilon \in \{0,1\}$ and for $\sigma=g-8-\epsilon>0$ with $\sigma$ even, $dim(\Lambda_n)=3g-2n-10>3g-2n-11-\epsilon=dim(Sec_{2g-2n-5-\sigma}(X))$. By Proposition 1.1 of \cite{LN}, we have that for the general extension $u$ in $\Lambda_n$, the corresponding vector bundle $E=E_u$ is stable, so we have a rational map 
$$\pi:\Lambda_n\to B^4(2,K_C),\pi(u):=[E_u],$$ 
\noindent and in particular $\pi(\Lambda_n)$ is non-empty.

\noindent Let $u\in\text{Ext}^1(K_C(-D_n),\mathcal O_C(D_n))$ be and let $\partial_u:H^0(C,K_C(-D_n))\to H^1(\mathcal O_C(D_n))$ be the corresponding coboundary map. By Serre duality we have that $\partial_u$ is a symmetric map, that is $\partial_u=\partial_u^{\vee}$, then $\text{ker}(\partial_u)=\text{ker}(\partial_u^{\vee})=(\text{Im}(\partial_u))^{\perp}$, in particular $\text{ker}(\partial _u)$ is uniquely determined, so the general fiber of the map $\pi_2|_{J_n}: J_n\to\mathbb P$ is irreducible and zero-dimensional, then $\Delta_n:=\overline{\Lambda_n}\subset\mathbb P$ is an irreducible component of $\mathcal W_3$ 
of dimension $3g-10-2n$.\end{proof}

\noindent Until here we construct the locus $\Delta_n$ consisting in extensions whose general element defines (through the map $\pi$) an stable rank two vector bundle with canonical determinant with four sections. To conclude the proof of the Theorem 3.1 it remains to study the image $\pi(\Delta_n)\subset B^4(2,K_C)$. By let the divisor $D_n$ to vary in all general effective divisors of degree $n$, in notation as in Section 2.1  for $d=2g-2, \delta=2g-2-n$ we construct the locus $\widetilde{\Delta_n}$ in the global space of extensions and we study the moduli map ${\pi}_{d,\delta}|_{\widetilde{\Delta_n}}:\widetilde{\Delta_n}\to B^4(2,K_C)$.

\noindent Let $u: 0\to\mathcal O_C(D_n)\to E_u\to K_C(-D_n)\to 0$  be a general extension in $\Delta_n$. Consider the quotient $E=E_u\to\!\!\to K_C(-D_n)$ and let $\Gamma_{u_n}=\Gamma$ the corresponding section of such quotient.  Let $S=\mathbb P(E)$ and 
$p:S\to C$ be the structure map. Since $\mathcal N_{\Gamma/S}\simeq K_C(-2D_n)$, 
then $h^0(\Gamma, \mathcal N_{\Gamma/S})=g-2n$ and $h^1(\Gamma, \mathcal N_{\Gamma/S})=1$. With this notation we have the following theorem

\begin{lem} (i). The couple $(S, \Gamma)$ is not obstructed, that is, the first-order infinitesimal deformations of the closed embedding $\Gamma\hookrightarrow S$ are unobstructed with $S$ not fixed, in particular $\Gamma$ is unobstructed in $S$ fixed
 and $\Gamma$ varies in a family $(g-2n)$-dimensional. 

\noindent (ii). The section $\Gamma$ is linearly isolated.

\noindent (iii). The section $\Gamma$ is specially isolated.
\end{lem}

\begin{proof} The proof of this result is similar to the proof in \cite{CR} (Lemma 2, Step 2.). In order to make the paper self-contained we show only
(iii).
 
\noindent Proof of (iii). Let $N=\mathcal O_C(D_n)$ be; tensoring by 
$N^{\vee}=\mathcal O_C(-D_n)$ the exact sequence
$$u: 0\to N\to E_u\to K_C\otimes N^{\vee}\to 0,$$
\noindent  we have in cohomology the exact sequence 
$$0\to H^0(\mathcal O_C)\to H^0(E\otimes N^{\vee})\to H^0(K_C\otimes (N^{\vee})^2)\xrightarrow{\delta}H^1(\mathcal O_C)\to\cdots$$
\noindent Since $|\mathcal O_S(\Gamma)|\simeq\mathbb{P}H^0(E(-D_n))$, the fact that $\Gamma$ is linearly isolated implies that $h^0(E(-D_n))=1$, then $\delta$ is an isomorphism onto its image. Denote by $<,>$ the Serre duality pairing, and let $V_1:=H^0(K_C\otimes (N^{\vee})^2)$ be. By (i) and (ii) and Section 2.2, we can identify $\delta(V_1)$ with the tangent space at $N$ to 
$\text{Quot}^C_E$, the Quot-Scheme that parametrizes quotient line bundles of $E$. Thus, to prove that $\Gamma$ is specially isolated we need to prove that 
$\delta(V_1)\cap T_{[N]}(W^0_3(C))=(0)=\delta(V_1)\cap (\text{Im}(\mu_{N}))^{\perp}$, where for 
$\ell=1,2$, $\mu_{N^{\ell}}:H^0(\mathcal O_C(\ell D_n))\otimes H^0(K_C(-\ell D_n))\to H^0(C,K_C)$ is the Petri map. Since $\text{Im}(\mu_{N^2})=V_1\subset\text{Im}(\mu_N)=H^0(K_C\otimes N^{\vee})$, then $(\text{Im}(\mu_N))^{\perp}\subset V_1^{\perp}$. 
Let $\delta(\omega)\in\delta(V_1)\cap(\text{Im}(\mu_{N}))^{\perp}\subset\delta(V_1)\cap V_1^{\perp}$. Then 
$<\delta(\omega),v>=0$ for all $v\in V_1$, that is, $\delta(\omega)\in\text{Ker}(V_1\xrightarrow{\beta}V_1^{\vee})$, 
where $\beta: x\to\beta_x, \beta_x(v)=<x,v>$. By duality, $\beta$ is an isomorphism, then $\delta(\omega)=0$; 
since $\delta$ is injective, $\omega=0$. This proves that $\Gamma$ is specially isolated.
\end{proof}

\begin{lem} The map $\pi: \Delta_n\to B^4(2,K_C)$ is generically injective.\end{lem} 

\begin{proof} For a general element $[E_u]\in\pi(\Delta_n)\subset B^4(2,K_C)$, the fiber $\pi^{-1}([E_u])$ 
corresponds to the extensions $u'\in\Delta_n$ such that there exists a diagram
\begin{align}
\xymatrix{
u:0\ar[r] & \mathcal O_C(D_n)\ar[r]^{\iota_1} & E_u\ar[r]\ar[d]_{\phi} & K_C(-D_n)\ar[r] & 0\\
u':0\ar[r] & \mathcal O_C(D_n)\ar[r]^{\iota_2} & E_{u'} \ar[r] & K_C(-D_n)\ar[r]  & 0
 }\end{align}
where $\phi:E_u\to E_{u'}$ is an isomorphism of stable bundles. The maps $\phi\circ\iota_1$ and $\iota_2$ 
determine two non-zero sections $\sigma_1\neq \sigma_2$ in $H^0(E_u(-D_n))$. For the general 
extension $u\in\Delta_n$, we have by Lemma 3.3 that  $h^0(E_u(-D_n))=1$, so $\phi\circ\iota_=\lambda\iota_2$, for some scalar $\lambda\neq 0$. Since $E_u$ is stable, then $E_u=E_{u'}$, i.e $u,u'$ are proportional in $\Delta_n$, then $\pi$ is generically injective. 

\noindent For a general effective divisor of degree $n$, $D_n$, $L=K_C(-D_n)$ depends on $\rho(L)=n$ parameters, then with notation as in Section 2.1 we have an irreducible component 
$\widetilde{\Delta}_n\subset\mathbb P(\mathcal E)$ of dimension $3g-10-2n+\rho(L)=3g-10-n$, where a point in
$\widetilde{\Delta_n}$ corresponds to the datum of a pair $(y,u)$ with $y=(\mathcal O_C(D_n),K_C(-D_n))$, $D_n$ a 
general effective divisor of degree $n$ and $u\in\Delta_n$ an extension as in Lemma 3.1. For $d=2g-2, \delta=2g-2-n$, we have that the moduli map ${\pi}_{d,\delta}|_{\widetilde{\Delta_n}}:\widetilde{\Delta_n}\to B^4(2,K_C)$ is generically injective and its image fills up an irreducible closed sublocus $\mathcal B_n:=\overline{{\pi}_{d,\delta}|_{\widetilde{\Delta_n}}(\widetilde{\Delta_n})} \subset B^4(2,K_C)$ of dimension $3g-10-n$.\end{proof} 
\noindent With the following Lemma we complete the proof of Theorem \ref{mainthm}.

\begin{lem}\label{lemaestratos}  For $n\geq 3$, $\mathcal B_{n+1}\subset\mathcal B_n$\end{lem}
\begin{proof}
\noindent Note that by an inductive step, it is enough to show that 
$\mathcal B_4\subset\mathcal B_3$. 

\noindent Consider a general element $E\in{\pi}_{d,\delta}|_{\widetilde{\Delta_4}}(\widetilde{\Delta_4})$, then there exists a general effective divisor $D_4$ of degree four such that $E$ fits into an exact sequence 
$$u:0\rightarrow \mathcal{O}_C(D_4)\rightarrow E \rightarrow K_C(-D_4)\rightarrow 0,$$ where $V:=Coker(\partial_u)$ is general in the Grassmannian $G(3,H^0(K_C(-D_4)))$. Fix three points in $Supp(D_4)$ and consider the degree three effective divisor $D_3$ given by these three points, then we can write $D_4=D_3+q$, $q\in C$ a general point.  The natural inclusion $H^0(K_C(-D_4))\hookrightarrow H^0(K_C(-D_3))$ inside $H^0 (K_C)$ induces an inclusion in Grassmannians $$G(3,H^0(K_C(-D_4)))\hookrightarrow G(3,H^0(K_C(-D_3))).$$ 

\noindent From Lemma 3.2 we have in particular that for $n=3,4$ the loci 
\small$$J_n:=\{W\times u| Im(\mu_W)\subset \{u=0\}\}\subset G(3,H^0(K_C(-D_n)))\times \mathbb{P}Ext^1(K_C(-D_n),\mathcal O_C(D_n)).$$
\normalsize
are such that $\pi_1|_{J_n}:J_n\rightarrow G(3,H^0(K_C(-D_n)))$ is dominant. Note also that for every $n\geq 3$, elements $u\in \Delta_n$ are such that $V:=Coker(\partial_u)$ satisfy the conditions of Proposition 3.1. Moreover, we have that
$\Delta_4=\overline{\pi_2(J_4)}\subsetneq\Delta_3=\overline{\pi_2(J_3)}$. By the construction from Section \ref{extensions} we have that $E\in\mathcal B_3$, then 
$\mathcal B_4\subset\mathcal B_3$.\end{proof}
 
 \noindent Note that the proof of Theorem 3.1 is obtained by applying from Lemma 3.2 to Lemma 3.5.

\end{document}